\documentclass[a4paper]{amsart}
\usepackage[utf8]{inputenc}
\usepackage[american]{babel}

\usepackage{amsmath,amsfonts,amssymb,amsthm,tabularx}  
\usepackage{mathbbol}
\usepackage{accents}
\usepackage{mathrsfs}
\usepackage{mathtools}
\usepackage{pifont}
\usepackage[usenames,dvipsnames]{color}

\usepackage[dvipsnames,x11names,rgb,table]{xcolor}
\usepackage[naturalnames=true]{hyperref}
\hypersetup{
  colorlinks = true,
  urlcolor = DeepSkyBlue4, linkcolor = Chartreuse4, citecolor = DarkOrange2
}

\usepackage{fixmath}
\usepackage{paralist}
\usepackage{stmaryrd}
\usepackage{longtable}
\usepackage{algorithm}
\usepackage{algpseudocode}

\usepackage{booktabs}
\usepackage{array}
\usepackage{csquotes}
\usepackage{diagbox}
\usepackage{afterpage}
\usepackage{pgf,tikz}
\usepackage{pgfplots}
\usetikzlibrary{arrows}
\usetikzlibrary{patterns}
\usetikzlibrary[positioning]
\usetikzlibrary[fit]
\usetikzlibrary{shapes.geometric}
\usetikzlibrary{calc}

\newtheorem{theorem}{Theorem}
\newtheorem*{theorem*}{Theorem}

\newtheorem{definition}{Definition}

\newtheorem{corollary}[theorem]{Corollary}
\newtheorem{lemma}[theorem]{Lemma}
\theoremstyle{remark}
\newtheorem{remark}[theorem]{Remark}

\newcommand\CC{{\mathbb C}}

\newcommand\Ct{{\mathbb C}(t)}

\newcommand{\xmark}{\ding{55}}%

\DeclareMathOperator\Pic{Pic}

\DeclareMathOperator\Cox{Cox}
\DeclareMathOperator\Bl{Bl}
\DeclareMathOperator\Cr{Cr}
\DeclareMathOperator\Gr{Gr}
\DeclareMathOperator\Spec{Spec}
\DeclareMathOperator\reg{reg}
\DeclareMathOperator\ord{ord}
\DeclareMathOperator\infor{in}

\usepackage{algorithmicx}
\usepackage{algorithm}
\usepackage{algpseudocode}

\MakeRobust{\Call}

\algdef{SE}[DOWHILE]{Do}{DoWhile}{\algorithmicdo}[1]{\algorithmicwhile\ #1}
\algrenewcommand{\algorithmiccomment}[1]{\hfill $\rhd$ \emph{#1}}
\algrenewcommand{\algorithmicrequire}{\textbf{Input:}}
\algrenewcommand{\algorithmicensure}{\textbf{Output:}}
\algnewcommand{\Or}{\textbf{or}}
\algnewcommand{\And}{\textbf{and}}
\algnewcommand{\Not}{\textbf{not}\,}
\algnewcommand\algorithmicforeach{\textbf{for each}}
\algdef{S}[FOR]{ForEach}[1]{\algorithmicforeach\ #1\ \algorithmicdo}

\definecolor{amber}{rgb}{1.0, 0.75, 0.0}
\definecolor{brightube}{rgb}{0.82, 0.62, 0.91}
\setdefaultitem{$\triangleright$}{}{}{}

\makeatletter
\@namedef{subjclassname@2020}{%
  \textup{2020} Mathematics Subject Classification}
\makeatother

\title{Discrete geometry of Cox rings of blow-ups of $\mathbb{P}^3$  }

\author{Mara Belotti \and Marta Panizzut} 

\address[Mara Belotti, Marta Panizzut]{
  Technische Universität Berlin,
  Chair of Discrete Mathematics/Geometry \\
  \texttt{\{belotti,panizzut\}@math.tu-berlin.de}
}

\begin{document}
\maketitle
\begin{abstract} We prove quadratic generation for the ideal of the Cox ring of the blow-up of  $\mathbb{P}^3$ at $7$ points, solving a conjecture of Lesieutre and Park. To do this we compute Khovanskii bases, implementing techniques which proved successful in the case of Del Pezzo surfaces. Such bases give us degenerations to toric varieties whose associated polytopes encode toric degenerations with respect to all projective embeddings. We study the edge-graphs of these polytopes and we introduce the \textit{Mukai edge graph}.
\end{abstract}

\vspace{0.5cm}
{\small \textup{2020} \textit{Mathematics Subject Classification}. 
14M25, 14Q15, 14D06, 52B05}
\section{Introduction}

Let $X_{k,d}$ be the blow-up of $k$ points in $d$-dimensional projective space. In \cite{Mukai04}, Mukai introduces a symmetric bilinear form on the Picard group  $\mathbb{Z}^{k+1}$ defined on the standard basis $H,E_1,E_2,\dots,E_k$ as follows:
\begin{equation} \label{eq:Mukaiform}  (H,H)=d-1 \,\,\,\, (H,E_i)=0 \,\,\,\, (E_i,E_j)=-\delta_{i,j}.
\end{equation}
A divisor class $D$ in $\Pic(X_{k,d})$ has anticanonical degree one if $\frac{1}{d-1}(-K_{k,d} \, , \,D)=1$, where $K_{k,d}$ is the canonical divisor of $X_{k,d}$. We consider the graph whose vertices are divisors of anticanonical degree one, and two vertices are connected by an edge if their  Mukai  product is zero. The $(k,d)$--\emph{Mukai edge graph} is the connected component containing $E_1$.  In this paper we begin the investigation of this graph as a combinatorial object encoding properties of the Cox ring of  $X_{k,d}$. 

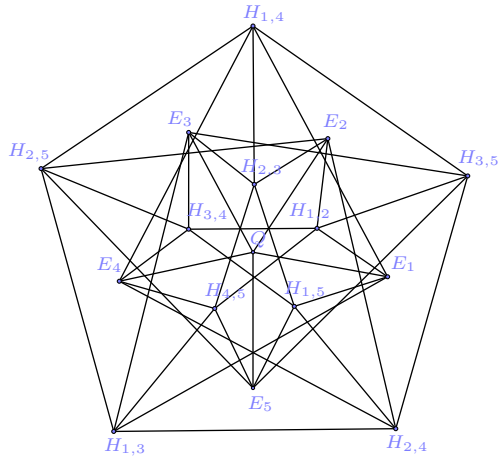
\begin{figure}[b]
\centering
\definecolor{xdxdff}{rgb}{0.49019607843137253,0.49019607843137253,1}
\begin{tikzpicture}[line cap=round, scale=0.3,line join=round,>=triangle 45,x=1cm,y=1cm]
\clip(-11,-10) rectangle (12,12);
\draw [line width=0.5pt] (-1.6787835177138586,-2.486299640158499)-- (0.056593700842840836,2.9994661446705666);
\draw [line width=0.5pt] (0.056593700842840836,2.9994661446705666)-- (1.8099704624455775,-2.392489691738408);
\draw [line width=0.5pt] (1.8099704624455775,-2.392489691738408)-- (-2.8216616990750865,1.0189333913326695);
\draw [line width=0.5pt] (-2.8216616990750865,1.0189333913326695)-- (2.8073119797726216,1.0578277025230167);
\draw [line width=0.5pt] (2.8073119797726216,1.0578277025230167)-- (-1.6787835177138586,-2.486299640158499);
\draw [line width=0.5pt] (-2.81957336327345,5.296225641824453)-- (0.056593700842840836,2.9994661446705666);
\draw [line width=0.5pt] (0.056593700842840836,2.9994661446705666)-- (3.2818785014231393,5.022875023519559);
\draw [line width=0.5pt] (3.2818785014231393,5.022875023519559)-- (2.8073119797726216,1.0578277025230167);
\draw [line width=0.5pt] (2.8073119797726216,1.0578277025230167)-- (5.901282997392197,-1.0839091219699954);
\draw [line width=0.5pt] (5.901282997392197,-1.0839091219699954)-- (1.8099704624455775,-2.392489691738408);
\draw [line width=0.5pt] (-2.8216616990750865,1.0189333913326695)-- (-2.81957336327345,5.296225641824453);
\draw [line width=0.5pt] (-2.8216616990750865,1.0189333913326695)-- (-5.861537129282881,-1.2815547128539626);
\draw [line width=0.5pt] (-5.861537129282881,-1.2815547128539626)-- (-1.6787835177138586,-2.486299640158499);
\draw [line width=0.5pt] (-1.6787835177138586,-2.486299640158499)-- (0,-6);
\draw [line width=0.5pt] (0,-6)-- (1.8099704624455775,-2.392489691738408);
\draw [line width=0.5pt] (-2.81957336327345,5.296225641824453)-- (0,0);
\draw [line width=0.5pt] (0,0)-- (3.2818785014231393,5.022875023519559);
\draw [line width=0.5pt] (0,0)-- (5.901282997392197,-1.0839091219699954);
\draw [line width=0.5pt] (0,0)-- (0,-6);
\draw [line width=0.5pt] (0,0)-- (-5.861537129282881,-1.2815547128539626);
\draw [line width=0.5pt] (-5.861537129282881,-1.2815547128539626)-- (0,10);
\draw [line width=0.5pt] (0,10)-- (5.901282997392197,-1.0839091219699954);
\draw [line width=0.5pt] (-6.104817881737462,-7.920302938071152)-- (-1.6787835177138586,-2.486299640158499);
\draw [line width=0.5pt] (-9.291568265803416,3.6968580121388066)-- (-6.104817881737462,-7.920302938071152);
\draw [line width=0.5pt] (-6.104817881737462,-7.920302938071152)-- (6.258832210404702,-7.7991678634326504);
\draw [line width=0.5pt] (-9.291568265803416,3.6968580121388066)-- (0,10);
\draw [line width=0.5pt] (0,10)-- (9.413965003886618,3.373019849570711);
\draw [line width=0.5pt] (9.413965003886618,3.373019849570711)-- (6.258832210404702,-7.7991678634326504);
\draw [line width=0.5pt] (1.8099704624455775,-2.392489691738408)-- (6.258832210404702,-7.7991678634326504);
\draw [line width=0.5pt] (2.8073119797726216,1.0578277025230167)-- (9.413965003886618,3.373019849570711);
\draw [line width=0.5pt] (0.056593700842840836,2.9994661446705666)-- (0,10);
\draw [line width=0.5pt] (-9.291568265803416,3.6968580121388066)-- (-2.8216616990750865,1.0189333913326695);
\draw [line width=0.5pt] (-9.291568265803416,3.6968580121388066)-- (3.2818785014231393,5.022875023519559);
\draw [line width=0.5pt] (6.258832210404702,-7.7991678634326504)-- (3.2818785014231393,5.022875023519559);
\draw [line width=0.5pt] (-2.81957336327345,5.296225641824453)-- (-6.104817881737462,-7.920302938071152);
\draw [line width=0.5pt] (-2.81957336327345,5.296225641824453)-- (9.413965003886618,3.373019849570711);
\draw [line width=0.5pt] (-5.861537129282881,-1.2815547128539626)-- (6.258832210404702,-7.7991678634326504);
\draw [line width=0.5pt] (0,-6)-- (-9.291568265803416,3.6968580121388066);
\draw [line width=0.5pt] (0,-6)-- (9.413965003886618,3.373019849570711);
\draw [line width=0.5pt] (-6.104817881737462,-7.920302938071152)-- (5.901282997392197,-1.0839091219699954);
\begin{scriptsize}
\draw [fill=xdxdff] (0,0) circle (2pt);
\draw[color=xdxdff] (0.235,0.58) node {$Q$};
\draw [fill=xdxdff] (0,10) circle (2.5pt);
\draw[color=xdxdff] (0.49,10.5) node {$H_{1,4}$};
\draw [fill=xdxdff] (0.056593700842840836,2.9994661446705666) circle (2.5pt);
\draw[color=xdxdff] (0.385,3.73) node {$H_{2,3}$};
\draw [fill=xdxdff] (-2.8216616990750865,1.0189333913326695) circle (2.5pt);
\draw[color=xdxdff] (-2,1.75) node {$H_{3,4}$};
\draw [fill=xdxdff] (2.8073119797726216,1.0578277025230167) circle (2.5pt);
\draw[color=xdxdff] (2.5,1.78) node {$H_{1,2}$};
\draw [fill=xdxdff] (1.8099704624455775,-2.392489691738408) circle (2.5pt);
\draw[color=xdxdff] (2.29,-1.67) node {$H_{1,5}$};
\draw [fill=xdxdff] (-1.6787835177138586,-2.486299640158499) circle (2.5pt);
\draw[color=xdxdff] (-1.19,-1.76) node {$H_{4,5}$};
\draw [fill=xdxdff] (-2.81957336327345,5.296225641824453) circle (2.5pt);
\draw[color=xdxdff] (-3.2,5.9) node {$E_3$};
\draw [fill=xdxdff] (3.2818785014231393,5.022875023519559) circle (2.5pt);
\draw[color=xdxdff] (3.625,5.77) node {$E_2$};
\draw [fill=xdxdff] (5.901282997392197,-1.0839091219699954) circle (2.5pt);
\draw[color=xdxdff] (6.6,-0.5) node {$E_1$};
\draw [fill=xdxdff] (0,-6) circle (2pt);
\draw[color=xdxdff] (0.325,-6.7) node {$E_5$};
\draw [fill=xdxdff] (-5.861537129282881,-1.2815547128539626) circle (2.5pt);
\draw[color=xdxdff] (-6.3,-0.56) node {$E_4$};
\draw [fill=xdxdff] (9.413965003886618,3.373019849570711) circle (2.5pt);
\draw[color=xdxdff] (9.91,4.12) node {$H_{3,5}$};
\draw [fill=xdxdff] (6.258832210404702,-7.7991678634326504) circle (2.5pt);
\draw[color=xdxdff] (6.76,-8.5) node {$H_{2,4}$};
\draw [fill=xdxdff] (-9.291568265803416,3.6968580121388066) circle (2.5pt);
\draw[color=xdxdff] (-9.8,4.42) node {$H_{2,5}$};
\draw [fill=xdxdff] (-6.104817881737462,-7.920302938071152) circle (2.5pt);
\draw[color=xdxdff] (-5.6,-8.6) node {$H_{1,3}$};
\end{scriptsize}
\end{tikzpicture}
\caption{The $(5,2)$--Mukai edge graph is the complement of the Clebsch graph represented here.}
\end{figure}

At the core of our analysis is the computation of toric degenerations of the rings $\Cox(X_{k,d})$. This approach has been initiated  by work of Sturmfles and Xu in  \cite{SturmfelsXu10}, where the quadratic generation of ideals of Cox rings of Del Pezzo surfaces $X_{k,2}$ is proven via degenerations to toric varieties using Khovanskii bases. Moreover, they use polytopes and polyhedral cones associated to the toric varieties to deduce  Ehrhart-type formulas for the multigraded Hilbert function of the Cox rings of $X_{k,2}$. 
Given a fixed number $k$ of points, the polytopes appearing from these toric degenerations are not always combinatorially equivalent. The authors focused on minimizing the number of their top dimensional faces and provide a classification for the case of five points in $\mathbb{P}^2$. This was not possible for higher numbers of points. In the case of the cubic surface, the polytope presented in \cite{SturmfelsXu10} already appeared in work of Todd \cite{Todd}, where he constructs polytopes associated to the general cubic surfaces which encodes the combinatorics of the lines.
 
The Mukai edge graph allows us to understand polytopes with minimal number of edges: for toric degenerations of $\Cox(X_{k,d})$ to normal varieties, polytopes having $(k,d)$-Mukai edge graph have minimal number of edges. In the case of the blow-up of six and  eight points in the plane, the vertex-edge graphs of the polytopes constructed in \cite{SturmfelsXu10} are Mukai edge graphs. Moreover, in the case of five points, the vertex-edge graph of the polytope which minimizes the number of facets among all the possible degeneration polytopes is the $(5,2)$--Mukai edge graph. We refer to Section \ref{sec:Mukai} and Proposition \ref{thm:MukaiGraph} for details.  

Khovanskii bases are an extremely important  computational tools in the study of toric degenerations, see for example \cite{KavehManon19}, \cite{BernalEtAl17}, \cite{DBGW20}, and they play a prominent role in our paper.  The Cox ring of  $X_{k,d}$ is finitely generated if and only if
\[\frac{1}{d+1}+\frac{1}{k-d-1}>\frac{1}{2},\]
as shown by Castravet and Tevelev  \cite{CastravetTevelev06} (we refer Table \ref{tab:fin_case} for a visual summary of the result). Moreover, since the conjecture posed by Batyrev and Popov in \cite{BatyrevPopov04}, a prominent question has been to  understand quadratic generation of the ideals of $\Cox(X_{k,d})$, see among others \cite{StillmanEtAl06}, \cite{ParkWon17}, \cite{LafaceVelasco09} and \cite{TestaEtAl09}.  There are only two finitely generated Cox rings  $\Cox(X_{k,d})$ for which quadratic generation was not yet proven: seven points in $\mathbb{P}^3$ and eight points in $\mathbb{P}^4$. We are interested in the first case, for which Lesieutre and Park showed that divisors of anticanonical degree one are explicit generators \cite{LesieutrePark17}. 

Our main result is given by the following theorem, solving a conjecture posed in \cite[Remark 4.4]{LesieutrePark17}.
\begin{theorem}\label{thm:main_res}
For generic choices of $7$ points in $\mathbb{P}(\mathbb{C}^4)$, the ideal $I_{7,3}$ of relations of $\Cox X_{7,3}$ is generated by quadrics.
\end{theorem}
We achieved this by showing that polynomials defining the generators form a Khovanskii  basis for the algebra $\Cox(X_{7,3})$. This Khovanskii  basis gives us a toric degeneration of $\Cox(X_{7,3})$ and the associated polytope is $10$--dimensional. We can project it to the the $7$--dimensional polytope underlying the effective cone of $X_{7,3}$. The projective embedding of $X_{7,3}$ corresponding to a given very ample line bundle degenerates to the toric surface associated with the $3$--polytope in the fiber of the projection: we obtain \textit{simultaneous degenerations} of $X_{7,3}$. 

Moreover, we compute a degeneration polytope with $126$ vertices and $2016$ edges such that its edge graph is given by the $(7,3)$--Mukai edge graph. We remark that this realization minimizes the number of edges but not the number of facets. Our computations are implemented in the computer algebra software \textsc{Oscar} \cite{OSCAR} and are available at 
\begin{center}
\url{https://github.com/marabelotti/7PointsPP3}.
\end{center} 

The paper is organized as follows. In Section \ref{sec:background} we introduce our notation and give an overview of known results about the Cox rings of $X_{k,d}$.  Khovanskii bases are introduced in  Section \ref{sec:Khovanskii_basis}. We limit ourselves to  the definition we need in our context (also referred to as sagbi basis), which is  less general then the one given in \cite{KavehManon19}. In Section \ref{sec:quad_gen} we prove quadratic generation of the ideal of $\Cox(X_{7,3})$ for generic choices of seven points in $\mathbb{P}^3$.  The univariate Hilbert function of $\Cox(X_{7,3})$ and  Ehrhart-type formulas for the multigraded one are given in Section \ref{sec:Hilbert_function}. Finally, Section \ref{sec:Mukai} collects observations and results about the Mukai edge graph.

\smallskip 

\textbf{Acknowledgements.}
We wish to thank Bernd Sturmfels for introducing us to the subject and we are grateful to him and to Michael Joswig for useful finalising comments. We are grateful to Matteo Gallet and Elisa Postinghel for useful discussions related to this project. We also thank Ulrich Derenthal and Jinhyung Park for clarifications on their works.

Mara Belotti is founded by the Deutsche Forschungsgemeinschaft (DFG, German Research Foundation) – Project-ID 286237555 – TRR 195

\section{Preliminaries on Cox Rings}
\label{sec:background}

We collect known results about Cox rings of varieties obtained by blowing up points in projective space. The section is organized following standard steps behind the study of these mathematical objects. 

\subsection{Cox rings and their gradings} 
Let $\mathbb{P}^d:=\mathbb{P}(\mathbb{C}^{d+1})$ be the $d$--dimensional projective space over the complex numbers, and
\[X_{k,d}:=\textnormal{Bl}_k \mathbb{P}^d\]
the blow-up of $k$ generic points in  $\mathbb{P}^d$. The Picard group $\Pic(X_{k,d})$  is a free abelian group of rank $k+1$ generated by the classes of the proper transforms of a general hyperplane $H$ and of the exceptional divisors $E_1,\dots,E_k$.  

The \textbf{Cox ring} of $X_{k,d}$ is the $\Pic(X_{k,d})$--graded ring
\[\Cox(X_{k,d}):=\bigoplus_{(m_0,\dots,m_k)\in \mathbb{Z}^d} H^0(X_{k,d},m_0H_0+m_1E_1+\dots+m_kE_k),\]
where multiplication is given by the standard multiplication of sections. This ring encodes important informations of the variety $X_{k,d}$ and it can be regarded as a generalization of its homogeneous coordinate ring: it contains all global sections of line bundles up to linear equivalence, so it encodes all possible embeddings of $X_{k,d}$ in projective spaces. 

 We look at the $\Pic(X_{k,d})$-grading on $\Cox(X_{k,d})$ as a \underline{$\mathbb{Z}^{k+1}$-grading} through the following identification:
\[m_0H-\sum_{i=1}^km_iE_i\in\Pic(X_{k,d}) \iff [m_0,m_0-m_1,\dots,m_0-m_k]\in\mathbb{Z}^{k+1}.\] 
The ring also has a natural $\mathbb{Z}$-grading compatible with the multigrading: given a $\Pic(X_{k,d})$--homogeneous element $s\in H^0(X_{k,d},D)$, we define $\deg(s):=\frac{1}{d-1}(-K_{k,d},D)$. In the expression, $K_{k,d}$ is the canonical divisor on $X_{k,d}$, 
 \[K_{k,d}=-(d+1)H+(d-1)\sum_{i=1}^{k}E_i,\] 
and $(\,\cdot \,,\, \cdot\, )$ is the Mukai symmetric bilinear form (\ref{eq:Mukaiform}) introduced in \cite{Mukai04}. This grading is usually referred to  as \underline{anticanonical degree}.
 
Using the aforementioned identification of $\Pic(X_{k,d})$ with $\mathbb{Z}^{k+1}$, we look at the \textbf{Weyl group} generated by the elements of the symmetric group $S_k$ permuting the last $k$ entries and by the \underline{Cremona element}
\[
\begin{small}
\Cr([m_0,m_1,\dots,m_k])=[m_0',\, m_1,\dots,\,m_{d+1},\,m_0'-m_0+m_{d+2},\dots,\,m_0'-m_0+m_{k}],\end{small}
\]
where $m_0'=\sum_{i=1}^{d+1}m_i-m_0$. 
Geometrically, the action of $S_k$ permutes the order of the blown up points, while the Cremona element is induced by  the Cremona transformation of $\mathbb{P}^d$ based at the first $d+1$ points. 
The action of the Weyl group preserves the Mukai bilinear form and  the dimension of  global sections. 

\subsection{Finite generation} The Cox ring of a $\mathbb{Q}$-factorial normal projective variety is finitely generated if and only if the variety is a Mori dream space \cite{MR1786494}.
In our case, this translates to conditions on $k$ and $d$.

\begin{theorem}{\cite[Theorem 1.3]{CastravetTevelev06}}\label{thm:fin_case} The $\mathbb{C}$--algebra $\Cox(X_{k,d})$ is finitely generated if and only if 
\[\frac{1}{d+1}+\frac{1}{k-d-1}>\frac{1}{2}.\]
In this cases, the projective variety $Z_{k,d}:=\mathbb{P}(\Cox(X_{k,d}))$ with respect to the anticanonical grading is a locally factorial, Cohen--Macaulay and Gorenstein scheme with rational singularities. Moreover, the anticanonical class is $-K_{Z_{k,d}}=\mathcal{O}_{Z_{k,d}}(s)$, where 
\[s=2(d+1)(k-d-1)\left(\frac{1}{d+1}+\frac{1}{k-d-1}-\frac{1}{2}\right)>0.\]
\end{theorem}

From now on, we will always assume $k\geq d+2$, since for $k\leq d+1$ the variety $X_{k,d}$ is a toric variety and its Cox ring is a polynomial ring.
Notice that for $d\geq 5$ we have finite generation only for $k\leq d+3$. We refer to Table~\ref{tab:fin_case}  for a visual summary of Theorem \ref{thm:fin_case}. 

\begin{center}
\begin{table}
\begin{tabular}{|c|c|c|c|c|c|c|}
\hline
\backslashbox{$d$}{$k$} & $d+2$ & $d+3$ & $d+4$ & $d+5$ & $d+6$\\
\hline
$2$ &  \checkmark & \checkmark & \checkmark & \checkmark & \checkmark \\
\hline
$3$ & \checkmark & \checkmark & \cellcolor{blue!25}\checkmark & &\\
\hline
$4$ &  \checkmark & \checkmark & \checkmark & &\\
\hline
$\geq 5$ & \checkmark & \checkmark &  & &\\
\hline
\end{tabular}
\vspace{0.5cm}
\caption{Cases in which the Cox ring is finitely generated.}
\label{tab:fin_case}
\end{table}
\end{center}

\subsection{Generators $G_{k,d}$} For every pair $(k,d)\neq (8,4)$, explicit generators of the Cox ring $\Cox(X_{k,d})$ have been described in \cite[Theorem 3.2]{BatyrevPopov04}, \cite[Theorem 1.2]{CastravetTevelev06},  \cite[Theorem 1.2]{LesieutrePark17}, giving the following theorem.

\begin{theorem} \label{lem:antic_gen} Let $(k,d)\neq (8,4)$ such that $\Cox(X_{k,d})$ is finitely generated. Sections of divisors of anticanonical degree one are generators for $\Cox(X_{k,d})$.
\end{theorem}

We denote the set of aforementioned sections with $G_{k,d}$. For $(k,d)\neq (8,2),(7,3)$, $(8,4)$, the multidegrees of the elements in $G_{k,d}$ are  exactly the elements in the orbit of $E_1$ under the Weyl action, see \cite{CastravetTevelev06}. 
For the other three cases, we refer to Table \ref{tab:notonlyone}. In  \cite[Remark 4.2]{LesieutrePark17}, it is conjectured that the theorem also holds  for $X_{8,4}:=\Bl_{8}\mathbb{P}^4$.

\begin{remark} The values in Table \ref{tab:notonlyone} for $(8,2)$ are taken from \cite{BatyrevPopov04}. For the cases $(7,3)$ and $(8,4)$, we compute the Weyl orbit of $E_1$ using \textsc{Oscar}. We know that the elements in this orbit generate the polyhedral cone of effective divisors with coefficients in $\mathbb{R}$, of which we can compute the Hilbert basis. Since the elements of this basis have anticanonical degree one, we found all.
For the dimension of the global sections we refer to \cite{MR3347179}.
\end{remark}

\begin{small}
\begin{center}
\begin{table}
\begin{tabular}{|c|c|c|c|c|}
\hline
$(k,d)$ &  Weyl representatives  &orbit size & $\dim_{\mathbb{C}}H^0(X_{k,d},D)$ & $|G_{k,d}|$ \tabularnewline [0.5ex] 
\hline
$(8,2)$  & $E_1$  & $240$  &  $1$  &  $242$\tabularnewline [0.5ex] 
& $-K_{8,2}$  & $1$ &  $2$ &\tabularnewline [0.5ex] 
\hline
\cellcolor{blue!25}$(7,3)$ & \cellcolor{blue!5}$E_1$ & \cellcolor{blue!5}$126$ &  \cellcolor{blue!5}$1$ & \cellcolor{blue!5}$129$\tabularnewline [0.5ex] 
\cellcolor{blue!25}& \cellcolor{blue!5}$-\frac{1}{2}K_{7,3}$  & \cellcolor{blue!5} $1$ &  \cellcolor{blue!5}$3$ & \cellcolor{blue!5} \tabularnewline [0.5ex] 
\hline
&  $E_1$   & $2160$ & $1$   & \tabularnewline [0.5ex] 
$(8,4)$& $2H-2E_1-\sum_{i=2}^8 E_i$  & $240$ &  $3$ & $2886$\tabularnewline [0.5ex] 
& $-K_{8,4}$  & $1$ &  $6$ & \tabularnewline [0.5ex] 
\hline
\end{tabular}
\vspace{0.5cm}
\caption{Generators of Cox rings.}
\label{tab:notonlyone}
\end{table}
\end{center}
\end{small}

\subsection{Ideal of relations $I_{k,d}$} We now  consider pairs $(k,d)\neq(8,4)$ for which $\Cox(X_{k,d})$ is finitely generated.  We have a surjection of graded ring 
\begin{equation}
\label{eq:func}\phi_{k,d}:\mathbb{C}[G_{k,d}]\to\Cox(X_{k,d})\end{equation}
where $\mathbb{C}[G_{k,d}]$ is the $\mathbb{Z}^{k+1}$--multigraded polynomial ring with variables indexed by elements of $G_{k,d}$. 
 Therefore, we have the following presentation 
\[\Cox(X_{k,d})\simeq \frac{\mathbb{C}[G_{k,d}]}{I_{k,d}},\]
where $I_{k,d}$ is the ideal of relations between sections of anticanonical degree one. 

\begin{theorem}[\cite{LafaceVelasco09},\cite{TestaEtAl09},\cite{SturmfelsXu10}] For $(k,d)\neq (8,4),(7,3)$ such that $\frac{1}{d+1}+\frac{1}{k-d-1}>\frac{1}{2}$, the ideal of relations $I_{k,d}$ is generated by quadrics $(I_{k,d})_2$ .
\end{theorem}

This theorem is a bundle of results by different authors, all proven using very different approaches that we will now summarize. For $k=d+2$ the Cox ring $\Cox(X_{k,d})$ is isomorphic to the homogeneous coordinate ring of the Grassmannian $\Gr(2,d+3)$ defined by the quadratic Pl\"ucker relations. When $d=2$, $X_{k,2}$ are Del Pezzo surfaces. In the paper \cite{StillmanEtAl06},  the authors prove quadratic generation for $k = 4,5$ and $6$ finding Gröbner bases invariant under certain monomial actions. This approach  involves a search for the right term order and it becomes computationally challenging for the blow-up of seven and eight points. The complete $d=2$ case is considered in \cite{LafaceVelasco09}, \cite{TestaEtAl09} where it is shown,  using purely theoretical arguments, that the second column of the multigraded Betti table has non-zero entry only for $b_{1,2}$.  In \cite{SturmfelsXu10}, the authors considered a different computational approach based on sagbi bases and proved quadratic generation for the complete case $d=2$,  and for $k=d+3$.

We are interested in the following conjecture, presented in \cite[Remark 4.4]{LesieutrePark17}: The ideal of relations $I_{7,3}$ is generated by quadrics $(I_{7,3})_2$.
The proof of this result will follow from Theorem \ref{theo:quadric_t}. 

\subsection{Hilbert functions of $Z_{k,d}$} In \cite{ParkWon17}, Park and Won compute the Hilbert functions of $\Cox(X_{k,2})$ as a graded ring with respect to the anticanonical grading.

\begin{theorem}[\cite{ParkWon17}] \label{thm:Hilb_func_3} The univariate Hilbert functions of $\Cox(X_{k,2})$ for $4 \leq k\leq 8$ are given in Table \ref{tab:Hilb_2}.
\end{theorem}
\begin{center}
\begin{table}[h]
\begin{tabular}{|c|c|}
\hline
$k$ & \text{Hilbert functions} \\
\hline
$4$ & \begin{small}${\begin{aligned}\frac{1}{6!}(5t^6 + 75t^5 + 455t^4 + 1425t^3 + 2420t^2 + 2100t + 720) \end{aligned}}$\end{small} \tabularnewline [1.5ex] 
\hline
$5$ & \begin{small} ${\begin{aligned}\frac{1}{7!}(34t^7 + 476t^6 + 2884t^5 + 9800t^4 + 20146t^3 + 25004t^2 + 17256t + 5040) \end{aligned}}$ \end{small} \tabularnewline [1.5ex] 
\hline
$6$ & \begin{small}${\begin{aligned}\frac{1}{8!}(372t^8 + 4464t^7 + & 25200t^6 + 86184t^5+193788t^4\\ + 291816t^3 + &284640t^2 + 161856t + 40320)  \end{aligned}}$ \end{small} \tabularnewline [0.5ex] 
\hline
$7$ & \begin{small}${\begin{aligned} \frac{1}{9!}(9504t^9 + 85536t^8 + 412992t^7 &+ 1294272t^6 + 2860704t^5\\ + 4554144t^4+5125248t^3 + &3863808t^2 + 1752192t + 362880)  \end{aligned}}$\end{small} \tabularnewline [0.5ex] 
\hline
$8$ & \begin{small}${\begin{aligned} \frac{1}{10!}(1779840t^{10} + 8899200t^9 &+ 32140800t^8 + 75168000t^7  + 137531520t^6+\\ 186883200t^5 + 191635200t^4 &+ 141696000t^3 + 74183040t^2 + 24624000t + 3628800) \end{aligned}}$\end{small} \tabularnewline [0.5ex] 
\hline
\end{tabular}
\vspace{0.5cm}
\caption{Hilbert functions of $\Cox(X_{k,2})$}
\label{tab:Hilb_2}
\end{table}
\end{center}
\vspace{-.5cm}
The multigraded Hilbert functions have been described in  \cite{SturmfelsXu10} with a combinatorial flavour. In particular, the multivariate functions are presented in terms of Ehrhart type formulas, i.e,  values are described as the number of lattice points inside prescribed polyhedra, and for $\Cox(X_{d+3,d})$ in terms of decorations of phylogenetic trees. We will get back to this in Section \ref{sec:Hilbert_function}.

\section{Khovanskii bases of Cox Rings}
\label{sec:Khovanskii_basis}
In this section, we follow the notation of \cite{SturmfelsXu10} and consider $(k,d)\neq (8,4)$ giving finite generation.  We work over $\mathbb{C}(t)$, as we want to use results and constructions from tropical geometry \cite{StuMac15}. Intuitively, this corresponds to consider families of points.

Let $P=(p_1,\dots,p_k)$ be the $k\times (d+1)$ matrix whose columns are points in $\mathbb{P}(\mathbb{C}(t)^{d+1})$. 
For every divisor $D:=m_0H-\sum_{i=1}^k m_iE_i$ with $m_0>0$ of anticanonical degree $1$, we consider $h^0(D):=\dim_{\mathbb{C}} H^0(X_{k,d},D)$ variables $x_{D,i}$.  Let $\Ct[G_{k,d}]$ be the polynomial ring with these variables.  Moreover,  we choose a basis $f_1^D,\dots,f_{h^0(D)}^D$ for the linear system of hypersurfaces of degree $m_0$ passing through the point $p_i$ with multiplicity $m_i$. We say that the points $p_1,\dots,p_k$ are in \textit{general position} if all the polynomials $f_i^D$  are non zero. This is a stronger condition than requiring $p_1,\dots,p_k$ to be generic. We define the $\Ct$-algebra map
\begin{align*}
\phi_{k,d}(t):\mathbb{C}(t)[G_{k,d}]&\to \mathbb{C}(t)[x_1,\dots,x_k,y_1,\dots,y_k]\\
x_{E_i}^i &\to x_i \\
x_{D,i}&\to f_i^D\left(P\cdot\frac{y}{x}\right)x_1^{m_0-m_1}\cdots x_r^{m_0-m_r}.
\end{align*}
where $\frac{y}{x}$ is the vector $(\frac{y_1}{x_1},..,\frac{y_k}{x_k})$.

When the points have coordinates in $\mathbb{C}$, we consider the map above as a $\mathbb{C}$--algebra morphism, $\phi_{k,d}:\mathbb{C}[G_{k,d}] \to \mathbb{C}[x_1,\dots,x_k,y_1,\dots,y_k]$. 
This might look as an abuse of notation considering the map (\ref{eq:func}) defined in the previous section, but Theorem \ref{lem:antic_gen} tells us that for points in general position, the two maps are the same under the isomorphism \[\textnormal{im}(\phi_{k,d})\cong \Cox(X_{k,d})\] described in \cite[Corollary 2.14]{SturmfelsXu10}.

Let us make all of this more explicit for the case  we are most interested in: $(k,d)=(7,3)$. For simplicity of notation, we define the map 
\begin{align*}
\mathcal{S}:\mathbb{C}(t)[z_0,z_1,z_2,z_3]\times \Pic(X_{7,3})&\to \mathbb{C}(t)[x_1,\dots,z_7,y_1,\dots,y_7]\\
\mathcal{S}(f(z_0,z_1,z_2,z_3),m_0H-\sum_{i=1}^k m_iE_i)&\to f\left(P\cdot\frac{y}{x}\right)x_1^{m_0-m_1}\cdots x_7^{m_0-m_7}.
\end{align*}
For $D$ and $f_i^D$ defined as above, we have that $\phi_{7,3}(x_{D,i}) = \mathcal{S}(f_i^D,D)$. 
\begin{lemma} \label{lem:gen} Given  a choice of seven points $P=(p_1,p_2,p_3,p_4,p_5,p_6,p_7)$ in $\mathbb{P}(\mathbb{C}(t)^4)$, the $129$ images of the elements in $G_{k,d}$ under $\phi_{7,3}(t)$ are given by the list below. 
\begin{enumerate}
\item[\fbox{$E_i$}] $7$ variables $x_1,x_2,x_3,x_4,x_5,x_6,x_7$.
\item[\fbox{$H_{ijk}$}] $35$ elements defined by \underline{hyperplanes}:
\[\mathcal{S}\left(h_{(i,j,k)},H-E_i-E_j-E_k\right),\]
where $h_{(i,j,k)}$ is an hyperplane passing through the points $p_i,p_j,p_k$, for every distinct $i,j,k\in [7]$.
\item[\fbox{$Q_{ij}$}] $42$ elements defined by \underline{quadrics}: 
\[\mathcal{S}\left(q_{(i,j)},2H-\sum_{s\neq i,j}E_s-2E_j\right),\]
where $q_{(i,j)}$ is a quadric passing through every point except for $p_i$ and with multiplicity two at $p_j$, for every distinct $i,j\in [7]$.
\item[\fbox{$C_{ijk}$}] $35$ elements defined by \underline{cubics}:
\[\mathcal{S}\left(c_{(i,j,k)},3H-E_i-E_j-E_k-2\sum_{s\neq i,j,k}E_s\right),\]
where $c_{(i,j,k)}$ is a cubic passing with multiplicity one at $p_i,p_j,p_k$ and with multiplicity two at the other points, for every distinct $i,j,k\in [7]$.
\item[\fbox{$Qr_{i}$}] $7$ elements defined by \underline{quartics}:
\[\mathcal{S}\left(qr_{i},4H-3E_i-2\sum_{s\neq i}E_s\right)\]
where $qr_{i}$ is a quartic passing through all points with multiplicity two and with multiplicity $3$ at $p_i$, for every $i\in [7]$.
\item[\fbox{$S_i$}] $3$ elements defined by  \underline{sections of ${-K_{7,3}}/{2}$}:
\[\mathcal{S}\left(s_1,2H-\sum_{k}E_k\right) \ \ \mathcal{S}\left(s_2,2H-\sum_{k}E_k\right) \ \ \mathcal{S}\left(s_3,2H-\sum_{k}E_k\right),\]
where $s_1,s_2$ and $s_3$ are $3$ independent quadrics passing through all seven points.
\end{enumerate}
\end{lemma}
Notice that the choice of generators is not unique: the hypersurfaces corresponding to fixed divisors are determined up to a constant, and the three sections of ${-K_{7,3}}/{2}$ are determined modulo a linear change of coordinates.

This point of view gives us explicit generators of the Cox ring as a subalgebra of a polynomial ring and  allow us  to deal computationally with the problem. If we want to show quadratic generation of the ideal of relations $I_{k,d}=\ker(\phi_{k,d}) \subset \mathbb{C}[G_{k,d}]$, the first computational approach would be to compute a Gröbner basis and check that it is generated by quadrics, as done in \cite{StillmanEtAl06} for $d=2$ and $k=4,5,6$. For increasing values of $k$ and $d$, this task is computationally challenging. Moreover, we have a more explicit understanding of the images of the maps $\phi_{7,3}$. This is why we are going to use \textit{Khovanskii bases}.

\begin{remark} 
In \cite{SturmfelsXu10}, the authors write some of the equations for the corresponding hypersurfaces considered in the article  in terms of the Plücker coordinates of the points. They are able to do this for hyperplanes and quadrics. Already for cubics finding these expressions becomes tricky.  Efficiently computing the equation of the quartics passing through a given number of points is a key algorithmic aspect of our project. As our computations will show, the coefficients can become massive univariate polynomials in $t$.
\end{remark}

\subsection{Khovanskii bases}
We now consider a $\mathbb{C}(t)$--subalgebra $R$ of a polynomial ring $\mathbb{C}(t)[z_1,\dots,z_m]$. The \underline{order} $\ord(f)$ of an element $f$ in $\mathbb{C}(t)[z_1,\dots,z_m]$ is the smallest integer $\ord(f)$  such that $t^{-\ord(f)}f$ has neither a pole or a zero. Using terminology from tropical geometry,  $\ord(f)$ is the tropicalization of $f$ evaluated at the zero vector ${\bf 0}$. The initial form of $f$ is given by the polynomial 
\[\textnormal{in}_{{\bf 0}}(f)=(t^{-\ord(f)}f)|_{t=0} \in \mathbb{C}[z_1,\dots,z_m].\]
This is the initial form of $f$ with respect to the zero weight. See \cite[Section 2.4]{StuMac15} for the general definition  of $\text{in}_{\bf w}(f)$ for a vector ${\bf w}$. For a collection $S$ of polynomials, we indicate with $\text{in}_{{\bf 0}}(S)$ the set of initial forms of the elements in $S$. 

\begin{definition} A {Khovanskii basis} for
$R$ is a set of generators $\mathcal{F}\subset R$ such that $\textnormal{in}_{{\bf 0}}(f)$ is a monomial for every $f\in\mathcal{F}$ and such that the $\CC$-algebra $\textnormal{in}_{{\bf 0}}(R)$ is generated by the initial forms in $\infor_{\bf 0}(\mathcal{F})$.
\end{definition}

\begin{remark} In the literature, the definition of Khovanskii basis is more general than the one we presented, see \cite{KavehManon19}. Our definition is sometimes referred to as sagbi basis (Subalgebra Analogue to Gröbner Basis for Ideals), a particular case of Khovanskii basis. Even though we are not considering the full general context in which these bases are defined, we will still use the name Khovanskii following \cite{BernalEtAl17, DBGW20}. 
\end{remark}

Consider now the situation where $\mathbb{C}(t)[z_1,\dots,z_m]$ is a $\mathbb{Z}^k$--multigraded polynomial ring and $\mathcal{F}$ a set of $\mathbb{Z}^k$--homogeneous generators of $R$.
We define the morphism
\[\begin{aligned}\phi:\mathbb{C}(t)[\mathcal{F}]&\to\mathbb{C}(t)[z_1,\dots,z_m]\\
x_{f}&\mapsto f\end{aligned}\]
where $\mathbb{C}(t)[\mathcal{F}]$ is a $\mathbb{Z}^k$--multigraded polynomial ring  with variables $x_{f}$ indexed by the elements $f\in\mathcal{F}$ and $\deg(x_f)=\deg(f)$.
Similarly, we define the graded morphism
\[\begin{aligned}\textnormal{in}(\phi):\mathbb{C}[\mathcal{F}]&\to\mathbb{C}[z_1,\dots,z_m]\\
x_{f}&\mapsto \infor_{{\bf 0}}(f)\end{aligned}\]
The following lemma and other important properties of Khovanskii bases are collected in Remark $3.1$ in \cite{SturmfelsXu10}. Since it is a useful criteria for detecting Khovanskii bases, we present its proof. 

\begin{lemma} \label{lem:crit_khov} Consider the weight ${\bf w}=(w_f)_{f\in\mathcal{F}}$ where $w_f=ord(f)$. If we have
\[\ker(\infor(\phi))\subseteq\infor_{{\bf w}}(\ker(\phi))\]
then the set $\mathcal{F}$ is a Khovanskii basis of $R$.
\end{lemma} 
\begin{proof}
We need to show that $\mathbb{C}[\infor_{\bf 0}(\mathcal{F})]=\infor_{\bf 0}(R)$. Since $\mathbb{C}[\infor_{\bf 0}(\mathcal{F})]\subseteq \infor_{\bf 0}(R)$,  it is sufficient to show that the multigraded Hilbert functions $\textnormal{Hil}_{\mathbb{Z}^m}$ of the two subalgebras coincide.

The inclusion $\infor_{{\bf w}}(\ker(\phi))\subseteq\ker(\infor(\phi))$ is a generalization of \cite[Lemma 11.3]{Stu96} and it holds for every choice of generators of $R$. Therefore we get the equality $\infor_{{\bf w}}(\ker(\phi))=\ker(\infor(\phi))$. We then have 
\[\textnormal{Hil}_{\mathbb{Z}^m}\left(\frac{\mathbb{C}(t)[\mathcal{F}]}{\ker(\phi)}\right)=
\textnormal{Hil}_{\mathbb{Z}^m}\left(\frac{\mathbb{C}[\mathcal{F}]}{\infor_{{\bf w}}(\ker(\phi))}\right)=
\textnormal{Hil}_{\mathbb{Z}^m}\left(\frac{\mathbb{C}[\mathcal{F}]}{\ker(\infor(\phi))}\right)\]
The left hand side is the Hilbert function of $R$, which is preserved when taking initial forms following \cite[Theorem 3.3]{BernalEtAl17}. The right hand side is the Hilbert function of $\mathbb{C}[\infor_{\bf 0}(\mathcal{F})]$.

%
%
%
%

\end{proof}
From the lemma, it is clear that Khovanskii bases define flat degenerations of $R$ to a toric varieties, see also  \cite[Lemma 2.4.14]{StuMac15}.

\begin{remark} Generalizing  the proof of \cite[Lemma 11.4]{Stu96} to our case is not straightforward. The valuation takes value in $\mathbb{Z}$ which is not maximum well-ordered. The property is also required in \cite[Theorem 2.17]{KavehManon19} and \cite[Proposition 1.3]{ConcaVallaHerzog96}. 
\end{remark}


\section{Quadratic generation for $7$ points in $\mathbb{P}^3$}
\label{sec:quad_gen}

In this section, we focus again on the case of $7$ points in $\mathbb{P}^3$ and we prove that the ideal $I_{7,3}$ of relations of $\Cox X_{7,3}$ is generated by quadrics,  as conjectured in  \cite[Remark 4.4]{LesieutrePark17}. The explicit expression of the map $\phi_{7,3}(t)$  is given in Lemma \ref{lem:gen}. Let us denote the set of generators of $\textnormal{im}(\phi_{7,3}(t))$ described by the lemma with $\phi_{7,3}(t)(G_{7,3})$.

\begin{theorem}\label{theo:quadric_t}
There exists a choice of $7$ points in general position in $\mathbb{P}(\mathbb{C}(t)^4)$ for which the elements $\phi_{7,3}(t)(G_{7,3})$  are a Khovanskii basis of $\textnormal{im}(\phi_{7,3}(t))$ and the ideal $\ker(\phi_{7,3}(t))$  is generated by quadrics.
\end{theorem}

\begin{proof} Consider the following matrix whose columns are seven points in $\mathbb{P}(\mathbb{C}(t)^4)$ 
\[P=\begin{pmatrix}
t^{59} &  t^{44}  &  t^{79} &   t^{20}  & t^{12}  &  t^{81}  & t^{36}\\
t^{8} &  t^{72} &  t^{49} &  t^{39} &  t^{58}  &  t^{23}  &  t^{64}\\
t^{44} &  t^{58} &  t^{12} &  t^{52}  &   t^{57}  &  t^{49}  &   t^{51}\\
t^{25} & t^{23}   &   t^{60} & t^{72}  &  t^{45}  & t^{51}  & t^{6}\\ \end{pmatrix}.\]
We construct the corresponding generators $\phi_{7,3}(t)(G_{7,3})$ using \textsc{Oscar}. More comments on implementations and computations are discussed after the proof. 
The initial forms of these generators are the monomials summarized below. We use the same labelling as in Lemma \ref{lem:gen}.
\[\begin{matrix} \textcolor{MidnightBlue}{E_1}  & \textcolor{MidnightBlue}{E_2} & \textcolor{MidnightBlue}{E_3} & \textcolor{MidnightBlue}{E_4} & \textcolor{MidnightBlue}{E_5} & \textcolor{MidnightBlue}{E_6} & \textcolor{MidnightBlue}{E_7} \\ x_1 &x_2 & x_3 & x_4 & x_5 & x_6 & x_7 \end{matrix}\]

\[\begin{matrix} \textcolor{purple}{H_{123}} & \textcolor{purple}{H_{124}} & \textcolor{purple}{H_{125}} & \textcolor{purple}{H_{126}} & \textcolor{purple}{H_{127}} & \textcolor{purple}{H_{134}} \\

x_4x_6x_7y_5 & x_5x_6x_7y_3 & x_4x_6x_7y_3 & x_3x_4x_7y_5 & x_4x_5x_6y_3 &  x_2x_5x_6y_7  \\

\textcolor{purple}{H_{135}} & \textcolor{purple}{H_{136}} & \textcolor{purple}{H_{137}} & \textcolor{purple}{H_{145}} & \textcolor{purple}{H_{146}} & \textcolor{purple}{H_{147}} \\ 
 
x_2x_4x_6y_7 & x_2x_4x_7y_5 & x_2x_4x_6y_5 & x_2x_3x_6y_7 & x_2x_5x_7y_3 & x_2x_5x_6y_3 \\

\textcolor{purple}{H_{156}} & \textcolor{purple}{H_{157}} & \textcolor{purple}{H_{167}} & \textcolor{purple}{H_{234}} & \textcolor{purple}{H_{235}} & \textcolor{purple}{H_{236}} \\ 

x_2x_4x_7y_3 & x_2x_4x_6y_3 & x_2x_3x_4y_5 & x_5x_6x_7y_1 & x_4x_6x_7y_1 & x_1x_4x_7y_5 \\

\textcolor{purple}{H_{237}} & \textcolor{purple}{H245} & \textcolor{purple}{H_{246}} & \textcolor{purple}{H_{247}} & \textcolor{purple}{H_{256}} & \textcolor{purple}{H_{257}} \\

x_4x_5x_6y_1 & x_1x_6x_7y_3 & x_1x_5x_7y_3 & x_3x_5x_6y_1 & x_1x_4x_7y_3 & x_3x_4x_6y_1\\

\textcolor{purple}{H_{267}} & \textcolor{purple}{H_{345}} & \textcolor{purple}{H_{346}} & \textcolor{purple}{H_{347}} & \textcolor{purple}{H_{356}} & \textcolor{purple}{H_{357}} \\

x_1x_4x_5y_3 & x_1x_2x_6y_7 & x_1x_2x_5y_7 & x_2x_5x_6y_1 & x_1x_2x_4y_7 & x_2x_4x_6y_1\\

\textcolor{purple}{H_{367}} & \textcolor{purple}{H_{456}} & \textcolor{purple}{H_{457}} & \textcolor{purple}{H_{467}} & \textcolor{purple}{H_{567}} \\
x_1x_2x_4y_5 & x_1x_2x_3y_7 & x_1x_2x_6y_3 & x_1x_2x_5y_3 & x_1x_2x_4y_3 

\end{matrix}\]

\[\begin{matrix}
\textcolor{Peach}{Q_{12}} & \textcolor{Peach}{Q_{13}} & \textcolor{Peach}{Q_{14}} & \textcolor{Peach}{Q_{15}} \\
x_1x_4x_5x_6x_7y_1y_3 & x_1x_2x_4x_5x_6y_1y_7 & x_1x_2x_3x_5x_6y_1y_7 & x_1x_2x_3x_4x_6y_1y_7 \\

\textcolor{Peach}{Q_{16}} & \textcolor{Peach}{Q_{17}} & \textcolor{Peach}{Q_{21}} & \textcolor{Peach}{Q_{23}} \\
x_1^2x_2x_4x_5y_3y_7 & x_1x_2x_4x_5x_6y_1y_3 & x_2^2x_3x_4x_6y_5y_7 & x_1x_2^2x_4x_6y_5y_7\\

\textcolor{Peach}{Q_{24}} & \textcolor{Peach}{Q_{25}} & \textcolor{Peach}{Q_{26}} & \textcolor{Peach}{Q_{27}} \\
x_1x_2^2x_5x_6y_3y_7 & x_1x_2^2x_4x_6y_3y_7 & x_1x_2^2x_3x_4y_5y_7 & x_1x_2^2x_4x_6y_3y_5 \\ 
 
\textcolor{Peach}{Q_{31}} & \textcolor{Peach}{Q_{32}} & \textcolor{Peach}{Q_{34}} & \textcolor{Peach}{Q_{35}} \\
x_2x_4x_5x_6x_7y_3^2 & x_1x_4x_5x_6x_7y_3^2 & x_1x_2x_5x_6x_7y_3^2 & x_1x_2x_4x_6x_7y_3^2 \\

\textcolor{Peach}{Q_{36}} & \textcolor{Peach}{Q_{37}} & \textcolor{Peach}{Q_{41}} & \textcolor{Peach}{Q_{42}} \\
x_1x_2x_4x_5x_7y_3^2 & x_1x_2x_4x_5x_6y_3^2 & x_2x_4^2x_6x_7y_3y_5 & x_3x_4^2x_6x_7y_1y_5\\

\textcolor{Peach}{Q_{43}} & \textcolor{Peach}{Q_{45}} & \textcolor{Peach}{Q_{46}} & \textcolor{Peach}{Q_{47}} \\
x_2x_4^2x_6x_7y_1y_5 & x_2x_4^2x_6x_7y_1y_3 & x_1x_2x_4^2x_7y_3y_5 & x_2x_3x_4^2x_6y_1y_5 \\

\textcolor{Peach}{Q_{51}} & \textcolor{Peach}{Q_{52}} & \textcolor{Peach}{Q_{53}} & \textcolor{Peach}{Q_{54}} \\
x_2x_4x_5x_6x_7y_3y_5 & x_3x_4x_5x_6x_7y_1y_5 & x_2x_4x_5x_6x_7y_1y_5 & x_2x_5^2x_6x_7y_1y_3 \\ 
 
\textcolor{Peach}{Q_{56}} & \textcolor{Peach}{Q_{57}} & \textcolor{Peach}{Q_{61}} & \textcolor{Peach}{Q_{62}} \\ 
x_1x_2x_4x_5x_7y_3y_5 & x_2x_3x_4x_5x_6y_1y_5 & x_2x_4x_5x_6^2y_3y_7 & x_4x_5x_6^2x_7y_1y_3  \\
   
\textcolor{Peach}{Q_{63}} & \textcolor{Peach}{Q_{64}} & \textcolor{Peach}{Q_{65}} & \textcolor{Peach}{Q_{67}} \\
x_2x_4x_5x_6^2y_1y_7 & x_2x_3x_5x_6^2y_1y_7 & x_2x_3x_4x_6^2y_1y_7 & x_2x_4x_5x_6^2y_1y_3 \\

\textcolor{Peach}{Q_{71}} & \textcolor{Peach}{Q_{72}} & \textcolor{Peach}{Q_{73}} & \textcolor{Peach}{Q_{74}} \\
x_2x_3x_4x_6x_7y_5y_7 & x_1x_4x_6x_7^2y_3y_5 & x_1x_2x_4x_6x_7y_5y_7 & x_1x_2x_5x_6x_7y_3y_7 \\

\textcolor{Peach}{Q_{75}} & \textcolor{Peach}{Q_{76}}  \\
x_1x_2x_4x_6x_7y_3y_7 & x_1x_2x_3x_4x_7y_5y_7 \\
\end{matrix}\]

\[\begin{matrix}
\textcolor{Fuchsia}{C_{321}} & \textcolor{Fuchsia}{C_{421}} & \textcolor{Fuchsia}{C_{431}} \\
x_1^2x_2^2x_4x_5x_6y_3^2y_7 & x_1x_2^2x_3x_4^2x_6y_1y_5y_7 & x_1x_2x_4^2x_5x_6x_7y_1y_3^2  \\

\textcolor{Fuchsia}{C_{432}} & \textcolor{Fuchsia}{C_{521}} & \textcolor{Fuchsia}{C_{531}} \\
x_1x_2^2x_4^2x_6x_7y_3^2y_5 & x_1x_2^2x_3x_4x_5x_6y_1y_5y_7 & x_1x_2x_4x_5^2x_6x_7y_1y_3^2 \\ 
 
\textcolor{Fuchsia}{C_{532}} & \textcolor{Fuchsia}{C_{541}} & \textcolor{Fuchsia}{C_{542}} \\
x_1x_2^2x_4x_5x_6x_7y_3^2y_5 & x_1x_2x_4^2x_5x_6x_7y_1y_3y_5 & x_1x_2^2x_4^2x_6x_7y_3y_5^2 \\  
 
\textcolor{Fuchsia}{C_{543}} & \textcolor{Fuchsia}{C_{621}} & \textcolor{Fuchsia}{C_{631}} \\
x_1x_2x_4^2x_5x_6x_7y_3^2y_5 & x_1x_2^2x_4x_5x_6^2y_1y_3y_7 & x_1x_2x_4x_5x_6^2x_7y_1y_3^2 \\ 

\textcolor{Fuchsia}{C_{632}} & \textcolor{Fuchsia}{C_{641}} & \textcolor{Fuchsia}{C_{642}} \\
x_1x_2^2x_4x_5x_6^2y_3^2y_7 & x_2x_4^2x_5x_6^2x_7y_1^2y_3 & x_2^2x_3x_4^2x_6^2y_1y_5y_7 \end{matrix}\]

\[\begin{matrix}

\textcolor{Fuchsia}{C_{643}} & \textcolor{Fuchsia}{C_{651}} & \textcolor{Fuchsia}{C_{652}} \\
x_2x_4^2x_5x_6^2x_7y_1y_3^2 & x_2x_4x_5^2x_6^2x_7y_1^2y_3 & x_2^2x_3x_4x_5x_6^2y_1y_5y_7 \\ 

\textcolor{Fuchsia}{C_{653}} & \textcolor{Fuchsia}{C_{654}} & \textcolor{Fuchsia}{C_{721}} \\
x_2x_4x_5^2x_6^2x_7y_1y_3^2 & x_2x_4^2x_5x_6^2x_7y_1y_3y_5 & x_1^2x_2^2x_4x_5x_6y_3y_7^2\\ 
     
\textcolor{Fuchsia}{C_{731}} & \textcolor{Fuchsia}{C_{732}} & \textcolor{Fuchsia}{C_{741}} \\
x_1^2x_2x_4x_5x_6x_7y_3^2y_7 & x_1x_2^2x_4x_5x_6x_7y_3^2y_7 & x_1x_2x_3x_4^2x_6x_7y_1y_5y_7 \\ 
 
\textcolor{Fuchsia}{C_{742}} & \textcolor{Fuchsia}{C_{743}} & \textcolor{Fuchsia}{C_{751}} \\ 
x_1x_2^2x_4^2x_6x_7y_3y_5y_7 & x_1x_2x_4^2x_6x_7^2y_3^2y_5 & x_1x_2x_3x_4x_5x_6x_7y_1y_5y_7 \\ 

\textcolor{Fuchsia}{C_{752}} & \textcolor{Fuchsia}{C_{753}} & \textcolor{Fuchsia}{C_{754}} \\ 
x_1x_2^2x_4x_5x_6x_7y_3y_5y_7 & x_1x_2x_4x_5x_6x_7^2y_3^2y_5 & x_1x_2x_4^2x_6x_7^2y_3y_5^2\\ 

\textcolor{Fuchsia}{C_{761}} & \textcolor{Fuchsia}{C_{762}} & \textcolor{Fuchsia}{C_{763}} \\ 
x_1x_2x_4x_5x_6^2x_7y_1y_3y_7 & x_1x_2^2x_4x_5x_6^2y_3y_7^2 & x_1x_2x_4x_5x_6^2x_7y_3^2y_7 \\

\textcolor{Fuchsia}{C_{764}} & \textcolor{Fuchsia}{C_{765}} \\
x_2x_3x_4^2x_6^2x_7y_1y_5y_7 & x_2x_3x_4x_5x_6^2x_7y_1y_5y_7 \\

\end{matrix}\]
 
\[\begin{matrix}
\textcolor{PineGreen}{Qr_1} & \textcolor{PineGreen}{Qr_2} &  \textcolor{PineGreen}{Qr_3} \\
x_1x_2^2x_4^2x_5x_6^2x_7y_3^2y_5y_7 & x_1x_2x_4^2x_5x_6^2x_7^2y_1y_3^2y_5 & x_1x_2^2x_4^2x_5x_6^2x_7y_1y_3y_5y_7\\

\textcolor{PineGreen}{Qr_4}&  \textcolor{PineGreen}{Qr_5}&  \textcolor{PineGreen}{Qr_6}\\
x_1x_2^2x_4x_5^2x_6^2x_7y_1y_3^2y_7 & x_1x_2^2x_4^2x_5x_6^2x_7y_1y_3^2y_7 & x_1^2x_2^2x_4^2x_5x_6x_7y_3^2y_5y_7\\
 
\textcolor{PineGreen}{Qr_7}\\ 
x_1x_2^2x_4^2x_5x_6^2x_7y_1y_3^2y_5
\end{matrix}\]

\[\begin{matrix}
\textcolor{Dandelion}{S_1} & \textcolor{Dandelion}{S_2} & \textcolor{Dandelion}{S_3} \\
x_2x_4x_5x_6x_7y_1y_3 & x_1x_2x_4x_5x_6y_3y_7 & x_1x_2x_4x_6x_7y_3y_5 \\
\end{matrix}\]

These are $129$ monomials involving only $11$ variables, since $y_2,y_4,y_6$ do not appear. We look at the matrix $M$ whose columns are indexed by the variables $x_i$ and $y_i$, whose rows are indexed by elements of $\phi_{7,3}(t)(G_{7,3})$ and whose elements are the exponents of the monomials above.  We can use the \texttt{markov} command of \textsc{4ti2} to obtain the toric ideal defined by $M$. It is an ideal in $129$ variables generated by $4220$ quadratic relations distributed over $883$ different multidegrees. 

Thanks to Lemma \ref{lem:crit_khov}, we only need to verify that for every multidegree $D$ the dimension
\[\dim_{\mathbb{C}(t)}(D):=\dim_{\mathbb{C}(t)}\left(\frac{\mathbb{C}(t)[G_{7,3}]}{\ker\,\phi_{7,3}(t)}\right)_D\]
is equal to 
 \[\textnormal{exp}(D):=\#\textnormal{monomials in degree }D-\#\textnormal{ relations in degree }D.\] 
When evaluating the coordinates of the points  in $P$ at a generic value $t_0$, we get a choice of points $P(t_0)$ in  \textit{general position} in $\mathbb{P}(\mathbb{C}^4)$. We can also specialize $\phi_{7,3}(t)$ to $t_0$ and obtain the $\mathbb{C}$--algebra morphism 
\[\phi_{7,3}(t_0):\mathbb{C}[G_{7,3}]\to\mathbb{C}[x_1,\dots,x_7,y_1,\dots,y_7],\]
whose image is $\Cox X_{7,3}$ where $X_{7,3}$ is the blow-up of $\mathbb{P}(\mathbb{C}^4)$ at the points $P(t_0)$. 

It follows that 
\[\dim_{\mathbb{C}}\,\Cox(X_{7,3})_D=\dim_{\mathbb{C}}\Big(\frac{\mathbb{C}[G_{7,3}]}{\ker\,\phi_{7,3}(t_0)}\Big)_D=\dim_{\mathbb{C}(t)}(D)\geq \textnormal{exp}(D).\] 
Since the points $P(t_0)$ are in general position, $\dim_{\mathbb{C}}\,\Cox(X_{7,3})_D$ is the dimension of a linear system and it is invariant under the Weyl action. In particular, there are only $3$ orbits that we need to consider, whose details are given in Table \ref{tab:dim2}.

\begin{table}[h]
\begin{center}
\begin{tabular}{|c|c|c|c|}
\hline
Degree $D$ & $\dim \Cox(X)_{D}$ & number of binomials & number of relations\\
\hline
$(1, 1, 0, 0, 1, 1, 1, 1)$ & $2$ & $5$ & $3$\\
\hline
$(2, 1, 2, 1, 1, 1, 1, 1)$ & $4$ & $19$ & $15$ \\
\hline
$(4,2,2,2,2,2,2,2)$ &  $7$ & $69$ & $63$\\
\hline
\end{tabular} 
\end{center}
\caption{Expected number of relations for elements in the three Weyl group orbits.}
\label{tab:dim2}
\end{table}

For each degree $D$, we verify that $\dim_{\mathbb{C}}\,\Cox(X_{7,3})_D=\textnormal{exp}(D)$. By Lemma \ref{lem:crit_khov}, we can conclude that the generators given in Lemma \ref{lem:gen} form a Khovanskii basis for $\textnormal{im}(\phi_{7,3}(t)) =\frac{\mathbb{C}(t)[G_{7,3}]}{\ker\,\phi_{7,3}(t)} $. 
Since $\textnormal{in}_{\bf w}(\ker(\phi_{7,3}(t)))$ is generated by quadrics, the same property holds for  $\ker(\phi_{7,3}(t))$.
\end{proof}
In our computations, we implemented a random search for points in $\mathbb{P}(\mathbb{C}(t)^4)$ such that the corresponding generators $\phi_{7,3}(t)(G_{7,3})$ are a Khovanskii basis for the algebra $\text{im}(\phi_{7,3}(t))$. Unlike the points exhibited in \cite[Theorems 5.1 and 6.1]{SturmfelsXu10},  the coordinates of  the points in our matrix $P$ have large exponents. We noticed that if we impose smaller upper bounds for the exponents in $P$, we would obtained   not generic enough points for satisfying all required properties of a Khovanskii basis. When allowing exponents of $t$ in a range $(1,100)$, the random search for a matrix of points $P$ such that the corresponding elements  $\phi_{7,3}(t)({G}_{7,3})$ are a Khoavanskii basis find $3$ values in $\approx 220$ iterations. The computations have been carried out in \textsc{Oscar} (version 0.10.0) and it takes $\approx 21$ minutes to verify that a Khovanskii basis is Khovanskii on a standard machine. More examples of points found in our search are available in Table \ref{tab:moreexamples}. We recall that our code is available at the  page linked in the introduction.

\begin{remark} Efficiently computing equations for the quartic hypersurfaces $Qr_{i}$ in $\mathbb{P}(\mathbb{C}(t)^4)$  was an interesting computational challenge. The naive approach of describing them as a determinant of a $35\times 35$ matrix does not work when the entries of the matrix $P$ are polynomials of $t$ with large exponents. In the code we exploit the following trick: for generic choices of points, among the quadratic relations of the Cox ring, we find 
\[c_1E_j\textbf{Qr}_{\bf i}+c_2S_0Q_{ji}+c_3S_1Q_{ji}+c_4S_1Q_{ji}+c_5Q_{jk}Q_{ki}\textnormal{ for } i\neq j\neq k. \]
The coefficients are unknown polynomials in the Pl\"ucker coordinates. We can therefore use the equations found for  $Q_{ji}$ and $S_{i}$ to get the equations of the quartics $Qr_i$: the only conditions we need to impose is the vanishing of the coefficients of those polynomials where the variables $y_i$ and $x_i$ appear more than once or appear together.

\end{remark}

We now get to the proof of quadratic generation of the ideal $I_{7,3}$.

\begin{proof}[proof of Theorem \ref{thm:main_res}]
The quadratic relations of the Cox ring for generic choices of points can be written as homogeneous polynomial functions in the entries of the matrix $P$ whose columns are given by the blown-up points. Since the multigraded Hilbert function of the Cox ring is the same for all points in general position (see Remark \ref{rem:Hilb_general}), we can prove quadratic generation for generic points as follows. For every multidegree $D$ consider the matrix $M(D)$ whose columns are indexed by monomials in degree $D$ and rows are indexed by all relations of degree $D$ obtained by multiplying a quadratic relation with a monomial. The entries of this matrix are polynomial functions in the entries of $P$. Since there exists a choice of points in general position for which the rank of this matrix is exactly the number of columns of $M(D)$ minus the value of the Hilbert function at $D$, this is true for generic choices of points.
\end{proof}

Before going to the next section, we would like to present another approach that turned out to be computationally infeasible. The Castelnuovo--Mumford regularity of $\Cox(X_{7,3})$ is $9$. Following \cite[Remark 4.4]{LesieutrePark17}, it is sufficient to show that the ideal $I_{7,3}$ has no minimal generators in degrees up to $8$. We can do this computationally, fixing a choice of points with coordinates in $\mathbb{C}$. For every $k\leq 8$ we choose a set of Weyl representatives for the multidegrees of anticanonical degree $k$. For each representative $D$ we can construct the matrix $M(D)$ as described in the above proof and we only need to verify that the rank of such matrix is given by the number of columns of $M(D)$ minus the value of the Hilbert function at $D$. 
We managed to do this up to degree $5$, however for higher degrees the sparse matrix $M(D)$ became too large (i.e. $10^5\times 5\cdot10^4$ already for degree $6$) and computing the rank required difficult manipolations, which we were not able to achieve for all multidegrees in anticanonical degree $6,7$ and $8$.

\section{Hilbert functions of the Cox ring of $X_{k,3}$}
\label{sec:Hilbert_function}

As mentioned in the introduction, one of the motivation for the study carried out in \cite{SturmfelsXu10} are Ehrhart-type formulas for the Hilbert function of $\Cox(X_{k,2})$. The authors  establish a correspondence between the dimensions of the various graded parts of $\Cox(X_{k,2})$ and the numbers of lattice points in slices of convex polyhedral cones or in polytopes. Before starting a similar discussion for the case of $\Cox(X_{7,3})$, we first compute explicitely the univariate Hilbert function of $\Cox(X_{7,3})$, using the degeneration of $\textnormal{im}(\phi_{7,3}(t))$ found in the previous section. 

\begin{remark}\label{rem:Hilb_general} All choices of points in general position share the same multigraded Hilbert function of their Cox ring. For every multidegree $D$ the dimension of $(\Cox X_{7,3})_D$ is the dimension of a linear system of surfaces in $\mathbb{P}^3$ through multiple points, see \cite[Remark 5.4]{VolLaf07}. This also holds for  non-nef divisors.
\end{remark}

\subsection{Univariate Hilbert function} 
In Theorem \ref{theo:quadric_t} we found a toric degeneration of the variety defined by $\ker(\phi_{7,3}(t))$ to the toric variety defined by $\textnormal{in}_{{\bf w}}(\ker(\phi_{7,3}(t))$.  We use this degeneration to determine the Hilbert polynomial of  $\Cox(X_{7,3})$.

\begin{theorem} \label{theo:hil_pol}
The univariate Hilbert polynomial of $\Cox(X_{7,3})$ for points in general position is given by the polynomial in Table \ref{tab:Hilb_3} for $k=7$.
 \end{theorem}
 
\begin{proof}
The toric variety in Theorem \ref{theo:quadric_t} is normal, as verified using  \textsc{Oscar}.
The Hilbert polynomial of this variety is therefore the Ehrhart polynomial of the polytope given by the convex hull of the exponents of the 129 initial forms. The latter is the polynomial above computed in \textsc{Oscar}. Thanks to \cite[Corollary 2.4.9]{StuMac15} this is also the Hilbert polynomial of the variety defined by $\ker(\phi_{7,3}(t))$, and therefore of $\Cox X_{7,3}$ with points in general position $P(t_0)$, see Remark \ref{rem:Hilb_general}.
\end{proof}

\begin{remark} We can replicate the same proof  for the Grassmannian $\Cox(X_{5,3})$ and for $\Cox(X_{6,3})$ using the choices of points  given in \cite{SturmfelsXu10}. For completion, their univariate Hilbert polynomials are also reported in Table \ref{tab:Hilb_3}. 
\end{remark}

Using similar reasoning as in \cite{ParkWon17}, we can show that the Hilbert polynomial and the Hilbert function coincide. We report below the results needed to generalize the argument in the case of $\Cox(X_{k,3})$, collecting on the way properties of these rings that might be of independent interest.

\begin{footnotesize}
\begin{table}
\begin{tabular}{|c|c|}
\hline
$k$ & \\
\hline
$5$ & \begin{small} ${\begin{aligned}\frac{1}{8!}(14t^8 &+ 336t^7 + 3444t^6 + 19656t^5 + 68166t^4 + 146664t^3 \\ &+ 190456t^2 + 135744t + 40320) \end{aligned}}$ \end{small} \tabularnewline [1.5ex] 
\hline
$6$ & \begin{small}${\begin{aligned}\frac{1}{9!}(496t^9 + &8928t^8 + 72240t^7 + 344736t^6 + 1068816t^5 + \\2231712t^4 + &3136400t^3 + 2857824t^2 + 1528128t + 362880)  \end{aligned}}$ \end{small} \tabularnewline [0.5ex] 
\hline
$7$ & \begin{small} ${\begin{aligned}\frac{1}{10!}\big(147984 x^{10} + 1479840 x^9 +  7451280 x^8 &+ 24094080 x^7 + 54743472 x^6 \\ 
 + 90588960 x^5 + 110014320 x^4 + 96467520 x^3 &+ 58071744 x^2 + 21427200 x + 3628800\big)\end{aligned}}$\end{small} \tabularnewline [0.5ex] 
\hline
\end{tabular}
\vspace{0.5cm}
\caption{Hilbert functions of $\Cox(X_{k,3})$}
\label{tab:Hilb_3}
\end{table}
\end{footnotesize}

\begin{lemma} 
The variety $Z_{k,3}:=\mathbb{P}(\Cox(X_{k,3}))$ embedded in $\mathbb{P}^{|G_{k,3}|}$ through the surjection $\phi_{k,d}$ in \eqref{eq:func} is projectively normal and \[\Cox(X_{k,3})=\bigoplus_{m\in\mathbb{Z}}H^0(Z_{k,3},\mathcal{O}_{Z_{k,3}}(m)).\]
Moreover, the ring $\Cox(X_{k,3})$ is Gorenstein.
\end{lemma}
\begin{proof} The first point is a consequence of the surjectivity of \eqref{eq:func}. The second point is a corollary of Theorem \ref{thm:fin_case}: $\Spec(X_{k,3})$ is the affine cone over $Z_{k,3}$ and it has rational singularities thanks to \cite[Theorem 1]{MR870733} and \cite[Lemma 3.8]{CastravetTevelev06}. This proves that  $\Cox(X_{k,3})$ is  Cohen-Macaulay. Since it is also factorial, it is Gorenstein by \cite[Exercise 21.21]{Book:Eisenbud95}.
\end{proof}

We have the following corollary whose proof is exactly as in  \cite[Corollaries 3.8-3.10]{ParkWon17}.

\begin{corollary}\label{cor:propHilbfunc}  The following properties hold:
\begin{enumerate}
\item[$\circ$]$h^i(Z_{k,3},\mathcal{O}_{Z_{k,3}}(m))=0$ for $\{0<i<k+3,\,m\in\mathbb{Z}\}$ and for $\{i=k+3,m\geq -15+2k\}$.
\item[$\circ$]The Hilbert function and the Hilbert polynomial of $\Cox(X_{k,3})$ coincide.
\item[$\circ$]The Castelnuovo-Mumford regularity $\reg(\Cox(X_{k,3}))=3k-12$.
\end{enumerate}
\end{corollary}

The informations collected so far  allow us to compute the Hilbert function of $\Cox(X_{k,3})$ in the same way as in \cite{ParkWon17}. We will not report here the entire argument for our case, but we  present the necessary information in Table \ref{tab:degrees23_6}. The entries are computed using the dimension formula in \cite[Theorem 5.3]{VolLaf07}. The table summarizes  representatives under the Weyl action of the multidegrees of anticanonical degree $2,3$ and $4$ in the case $\Cox(X_{7,3})$. The orbit dimensions where computed  using \textsc{Oscar}.

\begin{footnotesize}
\begin{table}[h]
\begin{tabular}{|c|c|c|c|c|}
\hline
\begin{minipage}{6em} \begin{center} Anticanonical degree $m$ \end{center}
\end{minipage} & Multidegree $D$ & $\dim\Cox(X)_{D}$ & $|$Orbit$|$ & nef\\
\hline
2  & $(0,0,0,0,0,0,0,2)$ & $1$ & 126& \xmark \\
 & $(0,0,0,0,0,0, 1, 1)$ & $1$ & 2016 &\xmark \\
  & $(1, 0, 0, 1, 1, 1, 1, 1)$ & $2$ & 756& \\
 & $(2, 2, 1, 1, 1, 1, 1, 1)$ & $4$ & 126 & \\
 & $(4, 2, 2, 2, 2, 2, 2, 2)$ & $7$ & 1 & \\
\hline
3 & $(0, 0, 0, 0, 0, 0, 0, 3)$ & $1$ & 126& \xmark\\
 & $(0, 0, 0, 0, 0, 0, 2, 1)$ & $1$ & 4032 &\xmark\\\
 & $(0,0,0,0, 0, 1, 1, 1)$ & $1$ & 10080 &\xmark\\\
 & $(1, 0, 0, 1, 1, 1, 1, 2)$ & $2$ & 7560 & \xmark\\\
 & $(1, 0, 1, 1, 1, 1, 1, 1)$ & $3$ & 4032 &\\
 & $(2, 1, 1, 1, 1, 1, 1, 3)$ & $1$ & 126 &\xmark\\\
 & $(2, 1, 1, 1, 1, 1, 2, 2)$ & $5$ & 2016 &\\
 & $(3, 0, 2, 2, 2, 2, 2, 2)$ & $4$ & 56 &\\
 & $(3, 1, 1, 2, 2, 2, 2, 2)$ & $7$ & 756 &\\
 & $(4, 2, 2, 2, 2, 2, 2, 3)$ & $10$ & 126&\\
 & $(6, 3, 3, 3, 3, 3, 3, 3)$ & $14$ & 1&\\
\hline
4 & $(0, 0, 0, 0, 0, 0, 0, 4)$ & $1$ & 126& \xmark\\
 & $(0, 0, 0, 0, 0, 0, 1, 3)$ & $1$ & 4032 &\xmark\\
 & $(0,0,0,0, 0, 0, 2, 2)$ & $1$ & 2016 &\xmark\\
 & $(0,0,0,0, 0, 1, 1, 2)$ & $1$ & 30240 &\xmark\\
 & $(0,0,0,0, 1, 1, 1, 1)$ & $1$ & 20160 &\xmark\\
 & $( 1, 0, 0, 1, 1, 1, 1, 3)$ & $2$ & 7560 & \xmark\\\
 & $ (1, 0, 0, 1, 1, 1, 2, 2)$ & $2$ & 30240 & \xmark \\
 & $(1, 0, 1, 1, 1, 1, 1, 2)$ & $3$ & 24192 &\xmark\\\
 & $(1, 1, 1, 1, 1, 1, 1, 1)$ & $4$ & 576 &\\
 & $(2, 0, 0, 2, 2, 2, 2, 2)$ & $3$ & 756 &\\
 & $(2, 0, 1, 1, 2, 2, 2, 2)$ & $4$ & 12096 &\\
 & $(2, 1, 1, 1, 1, 1, 1, 4)$ &  $4$ & 126 & \xmark\\
 & $(2, 1, 1, 1, 1, 1, 2, 3)$ &  $5$ & 4032 & \xmark\\
 & $(2, 1, 1, 1, 1, 2, 2, 2)$ & $6$ &  10080 &\\
 & $(3, 0, 2, 2, 2, 2, 2, 3)$ & $5$ &  1512 &\\
 & $(3, 1, 1, 2, 2, 2, 2, 3)$ & $8$ &  7560 &\\
 & $(3, 1, 2, 2, 2, 2, 2, 2)$ & $10$ &  4032 &\\
 & $(4, 2, 2, 2, 2, 2, 2, 4)$ & $11$ &  126 &\\
 & $(4, 2, 2, 2, 2, 2, 3, 3)$ & $13$ &  2016&\\
 & $(5, 1, 3, 3, 3, 3, 3, 3)$ & $12$ &  56 &\\
 & $(5, 2, 2, 3, 3, 3, 3, 3)$ & $16$ &  756 &\\
 & $(6, 3, 3, 3, 3, 3, 3, 4)$ & $20$ &  126  &\\
 & $(8, 4, 4, 4, 4, 4, 4, 4)$ & $25$ &  1 &\\
 \hline
 \end{tabular}
 \vspace{0.5cm}
 \caption{Weyl representatives and their properties in anticanonical degree $2$, $3$ and $4$ for $\Cox(X_{7,3})$.}
 \label{tab:degrees23_6}
\end{table}
\end{footnotesize}

\subsection{Ehrhart-type formula} Now we deduce Ehrhart-type formula for the Hilbert function of $\text{Cox}(X_{7,3})$ as in \cite{SturmfelsXu10}. More precisely, we want to express the counting function 
\[
\psi: \mathbb{Z}^8 \rightarrow \mathbb{Z}, \ \ D \mapsto \text{dim}_{\CC} \, (\text{im}(\phi_{7,3})_D)
\]
as counting function of the number of lattice points in a polytope.  We look at the $129$ initial forms of the elements in $\phi_{7,3}(t)(G_{7,3})$ listed in the proof of Theorem \ref{theo:quadric_t}. As we already remarked, they involve only eleven variables, since $y_2, y_4,$ and $y_6$ do not appear. Therefore, any monomial in the subalgebra $\text{im}(\phi_{7,3}(t))$ is of the form 
\[
x_1^{a_1}x_2^{a_2}x_3^{a_3}x_4^{a_4}x_5^{a_5}x_6^{a_6}x_7^{a_7} y_1^{b_1}y_3^{b_3}y_5^{b_5}y_7^{b_7}, 
\]
corresponding to a point $(a_1, \dots, a_7, b_1,b_3, b_5,b_7) \in \mathbb{Z}^{11}$. We apply the unimodular transformation $(a,b) \mapsto (m_0, m_1, \dots, m_7, x,y,z)$ in $\mathbb{Z}^{11}$ defined as follows:
\begin{eqnarray*}
m_0 = b_1+b_3+b_5+b_7, \ \ m_1 = a_1 +b_1, \ \ m_2 = a_2, \ \ m_3 = a_3+b_3, \ \ m_4  = a_4, \\
m_5 = a_5+b_5, \ \  m_6= a_6, \ \   m_7 = a_7+b_7, \ \ x = b_3, \ \ y = b_5 ,\ \  z=b_7 . 
\end{eqnarray*}
Let $\Gamma$ be the $11$--dimensional cone in $\mathbb{R}^{11}$ generated by the $129$ lattice points $(m_0, m_1, \dots, m_7, x,y,z)$ which are images of the $129$ initial forms.   The counting function equals 
\[
\psi(m_0, m_1, \dots, m_7, x,y,z) = \#\{(x,y,z) \in \mathbb{Z}^3 \,:\, (m_0, m_1, \dots, m_7, x,y,z)\in \Gamma\}.
\]
Using \textsc{Oscar} we can analyze $\Gamma$: its f-vector is 
\[
(126, \,  2016, \, 10839, \,  27198, \,  37635, \,  30891, \,  15286, \,  4416, \, 683, \, 48),
\]
and get the $48$ inequalities describing the facets. The projection of the cone $\Gamma$ onto the first $8$ coordinates gives the effective cone of $X_{7,3}$. For every very ample divisor in the effective cone, the toric variety associated to the $3$--dimensional polytope in the fiber of the projection is a toric degeneration of the embedding of $X_{7,3}$ defined by the ample divisor.

\section{Mukai edge graphs}
\label{sec:Mukai}
The matrix $P$ presented in the proof of Theorem \ref{theo:quadric_t} is not the first one we found in the random search. We looked  for a specific choice of points such that the vertex-edge graph of the polytope associated to the toric degeneration is the $(7,3)$--Mukai edge graph defined in the introduction. 

As we have seen, one of the perk of  Khovaskii bases is that they allow us to obtain Ehrhart-type formula for the multigraded Hilbert function of the Cox ring. The latter is interesting as it gives dimensions of linear systems of hypersurfaces through multiple points. Via Ehrhart-type formulas determining these dimensions boils down to counting points in a polytope. This process becomes more efficient if we get a better understanding of the combinatorics of the polytopes and the number of its faces. In \cite{SturmfelsXu10} the authors  address the problem of finding degeneration polytopes with  minimal number of facets and give conjectures in their cases. In this section we introduce the study of the Mukai edge graph as combinatorial object encoding interesting properties of the degeneration. 

From now on we will refer to the $(k+d)$--dimensional polytopes associated to \underline{normal} toric degenerations of $\textnormal{im}(\phi_{k,d}(t))$ for choices of points in general position as {\bf  $(k,d)$--Khovanskii degeneration (KD)} polytopes.  In general,  toric degenerations of a normal variety are not necessarily normal. We wonder whether this is true for  toric degenerations defined by Khovanskii bases. 

\begin{lemma} Let $\mathcal{P}$ be a  $(k,d)$--KD polytope. The lattice points of $\mathcal{P}$ corresponding to generators in $G_{k,d}$ whose multidegree is in the Weyl orbit of $E_1$ are vertices. Moreover, for $(k,d)\neq (8,4)$, $(k,d)$--KD polytopes do not have interior lattice points. For $(k,d)= (8,4)$, they have exactly one interior lattice point.
\end{lemma}
\begin{proof}
 Since the  toric variety $X_{\mathcal{P}}$ is normal, its Hilbert polynomial coincides with the Ehrhart polynomial of $\mathcal{P}$. The variety $X_{\mathcal{P}}$ is a toric degeneration of $\textnormal{im}(\phi_{k,d}(t))$ for a choice of points in general position in $\mathbb{P}^d(\mathbb{C}(t))$. As explained in the proof of Theorem \ref{theo:hil_pol}, the  Ehrhart polynomial of $\mathcal{P}$ is the Hilbert polynomial of $\Cox(X_{k,d})$. So the number of lattice points of $\mathcal{P}$ and of generators $G_{k,d}$ coincide. 

We now show that the lattice points associated to divisors in the Weyl orbit of $E_1$ are vertices of $\mathcal{P}$. Any such divisor $D$ has a unique  associated generator in $G_{k,d}$ and lattice point $P_D$ in $\mathcal{P}$. Suppose by contradiction that 
\[t_DP_D=\sum_{E} t_E P_E\hspace{0.3cm}\textnormal{where}\hspace{0.3cm}\sum_{E}t_E=t_D,\hspace{0.2cm}t_E\in\mathbb{Z}_{\geq 0},\]
where the sum is taken over divisors $E\neq D$  of anticanonical degree one and $P_E$ are the lattice points associated to them. Such a relation between lattice points give the following relation in the toric ideal of $\mathbb{C}(t)[G_{k,d}]$:
\[x_D^{t_D}-\prod_E x_E^{t_E}.\]
This implies that $t_D(D,D)=\sum_E t_E (E,D)$,
but $(D,D)=-1$, while $(E,D)\geq 0$ for every $E\neq D$. It also follows that for $(k,d)\neq(8,2), (7,3)$, the polytope is reflexive.

We have only two cases in which the generators have sections associated to divisors not in the Weyl orbit of $E_1$. These corresponds to $(8,2)$, namely Del Pezzo surfaces of degree one, and the case of our interest $(7,3)$. Here we can compute the number of interior lattice points from the Ehrhart polynomial $\textnormal{Ehr}(\mathcal{P})$ of the polytope  by evaluating $(-1)^{10}\textnormal{Ehr}(\mathcal{P})(-1)$.
\end{proof}

We are now ready to prove that the Mukai edge graph minimizes the number of edges of KD polytopes. We begin with a short technical remark.

\begin{remark}\label{rem:nef} Given a divisor $D$ of anticanonical degree one and a multidegree $E$ such that $(E,D)<0$, all monomials of $\mathbb{C}[G_{k,d}]$ appearing in degree $E$ are divisible by $x_D$. We can prove it by induction on the degree of $E$. In fact, suppose there exists a multidegree $E'$ and a divisor $D'$ of anticanonical degree one such that $D'+E'=E$. If $(E',D)<0$ then by induction all monomials in $E'$ are divisible by $x_D$ and we have the statement. Otherwise, $(E',D)\geq 0$ which means that there cannot be a $D'\neq D$ such that $(D'+E',D)< 0$.

\end{remark}

\begin{center}
\begin{table}
\begin{tabular}{|c |c|}
\hline
\rule{0pt}{3\normalbaselineskip}
$\begin{pmatrix}
t^{59} &  t^{44}  &  t^{79} &   t^{20}  & t^{12}  &  t^{81}  & t^{36}\\
t^{8} &  t^{72} &  t^{49} &  t^{39} &  t^{58}  &  t^{23}  &  t^{64}\\
t^{44} &  t^{58} &  t^{12} &  t^{52}  &   t^{57}  &  t^{49}  &   t^{51}\\
t^{25} & t^{23}   &   t^{60} & t^{72}  &  t^{45}  & t^{51}  & t^{6}\\ \end{pmatrix}$ & \small{$[126,2016,10839,\cdots,4416,683,48]$}\\
\hline
\rule{0pt}{3\normalbaselineskip}
$\begin{pmatrix}
t^{13} &  t^{68}  &  t^{70} &   t^{27}  & t^{68}  &  t^{90}  & t^{10}\\
t^{45} &  t^{32} &  t^{9} &  t^{29} &  t^{15}  &  t^{13}  &  t^{61}\\
t^{45} &  t^{82} &  t^{26} &  t^{22}  &   t^{9}  &  t^{54}  &   t^{45}\\
t^{2} & t^{36}   &   t^{32} & t^{25}  &  t^{16}  & t^{87}  & t^{68}\\ \end{pmatrix}$ & \small{$[126, 2077, 11481, \cdots , 5661, 916, 64]$}\\
\hline
\rule{0pt}{3\normalbaselineskip}
$\begin{pmatrix} 
t^{96} &  t^{77}  &  t^{27} &   t^{21}  & t^{20}  &  t^{63}  & t^{76}\\
t^{35} &  t^{98} &  t^{80} &  t^{81} &  t^{60}  &  t^{85}  &  t^{53}\\
t^{23} &  t^{22} &  t^{76} &  t^{53}  &   t^{99}  &  t^{31}  &   t^{97}\\
t^{27} & t^{84}   &   t^{30} & t^{79}  &  t^{56}  & t^{21}  & t^{74}\\ \end{pmatrix}$ & \small{$[126, 2053, 11198, \cdots , 5359, 870, 60]$}\\
\hline
\rule{0pt}{3\normalbaselineskip}
$\begin{pmatrix} 
t^{58} &  t^{100}  &  t^{61} &   t^{11}  & t^{18}  &  t^{40}  & t^{100}\\
t^{23} &  t^{38} &  t^{40} &  t^{89} &  t^{63}  &  t^{56}  &  t^{66}\\
t^{92} &  t^{25} &  t^{76} &  t^{27}  &   t^{67}  &  t^{36}  &   t^{73}\\
t^{18} & t^{90}   &   t^{92} & t^{73}  &  t^{54}  & t^{8}  & t^{18}\\ \end{pmatrix}$ & \small{$[126, 2052, 11163, \cdots , 4335, 657, 46]$}\\
\hline
\addlinespace[2ex]
\end{tabular}
\caption{Different choices of points and the $f$--vector of their associate $10$--dimensional $(3,7)$--KDP polytope}
\label{tab:moreexamples}
\end{table}
\end{center}

\begin{theorem} \label{thm:MukaiGraph} The $(k,d)$--Mukai edge graph is a subgraph of the vertex-edge graph of any $(k,d)$--KD polytope.  Therefore,  a $(k,d)$--KD polytope with minimal number of edges has the $(k,d)$--Mukai edge graph as vertex-edge graph.
\end{theorem}
\begin{proof}
We fix a divisor $D$ in the Weyl orbit of $E_1$. Let $P_D$ be the lattice point associated to the generator corresponding to $D$. Consider now any divisor $E$ of anticanonical degree one such that $(E,D)=0$. Suppose by contradiction that the $P_E-P_D$ is not a generating ray of the vertex cone at $P_D$. We have
\[t_E(P_E-P_D)=\sum_{E'}t_{E'}(P_{E'}-P_D),\hspace{0.5cm}t_{E'},t_E\in\mathbb{Z}_{\geq 0}.\]
This translates into the relation
\[x_E^{t_E}x_D^{t_N-t_E}=\prod_{E'}x_{E'}^{t_{E'}}x_D^{t_N-\sum{t_{E'}}}, \hspace{0.5cm}t_N\in\mathbb{Z}_{\geq 0}.\]
The multidegree of the relation is $t_E E+(t_N-t_E)D$, and $(t_E E+(t_N-t_E)D,D)<0$. By Remark \ref{rem:nef}, we can divide the relation by $x_D^{t_N-t_E}$, obtaining 
\[x_E^{t_E}=\prod_{E'}x_{E'}^{t_{E'}}x_D^{t_E-\sum{t_{E'}}} \hspace{0.5cm}t_N\in\mathbb{Z}_{\geq 0}.\]
By the previous proof we know that this is not possible.
\end{proof}

We conclude by comparing the Mukai edge graph with  the problem of minimizing the number of facets of $(k,2)$--KD polytopes in the case of Del Pezzo surfaces.  The proof of the 
following lemma is purely computational. The code is available at \url{https://github.com/marabelotti/7PointsPP3}.

\begin{lemma} Let $P_{k,2}$ denote the $(k,2)$--KD polytopes obtained in \cite{SturmfelsXu10}.
\begin{enumerate}
\item[$\circ$] The unique $(5,2)$--KD polytope having the $(5,2)$--Mukai edge graph is the polytope with the minimum number of facets.
\item[$\circ$] For $k\in\{6,8\}$, $P_{k,2}$ has the $(k,2)$--Mukai edge graph as vertex-edge graph.
\end{enumerate}
\end{lemma}

The polytopes $P_{k,2}$ in \cite{SturmfelsXu10} are conjectured to minimize the number of facets. At first we conjectured that a $(k,d)$--KD polytope with  $(k,d)$--Mukai edge graph had the smallest number of facets among all $(k,d)$--KD polytopes. However, this is not true: the polytope presented in Theorem \ref{theo:quadric_t} has the $(7,3)$--Mukai edge graph, but the number of facets is not minimal. In Table \ref{tab:moreexamples}, we give other choices of points together with the f--vector of their corresponding $(7,3)$--KD polytope.

\bibliographystyle{alpha}
\bibliography{7PointsPP3.bib}
\end{document}